\documentclass[12pt,a4paper]{amsart}
\usepackage{amsfonts}
\usepackage[active]{srcltx}

\usepackage{enumerate}
\usepackage{amsmath,amssymb,xspace,amsthm}

\newtheorem{theorem}{Theorem}[section]

\newtheorem{remark}[theorem]{Remark}
\newtheorem{lemma}[theorem]{Lemma}
\newtheorem{proposition}[theorem]{Proposition}

\numberwithin{equation}{section} \theoremstyle{definition}

\def\span{\operatorname{span}}

\newcommand{\ZZ}{{\mathbb Z}}

\newcommand{\C}{\ensuremath{\mathbb C}\xspace}
\newcommand{\Q}{\ensuremath{\mathbb{Q}}\xspace}
\renewcommand{\a}{\ensuremath{\alpha}}

\newcommand{\h}{\ensuremath{\mathfrak{h}}}

\newcommand{\Z}{\ensuremath{\mathbb{Z}}\xspace}

\newcommand{\W}{\ensuremath{\mathcal{W}}\xspace}

\renewcommand{\phi}{\varphi}

\renewcommand{\leq}{\leqslant}
\renewcommand{\geq}{\geqslant}

\def\mu{\mathfrak{u}}

\def\sl{\mathfrak{sl}}
\def\gl{\mathfrak{gl}}

\def\span{\text{span}}

\def\bE{\bar E}

\begin{document}
\title[Irreducible Witt modules]{New families of irreducible weight  modules over $\sl_{3}$}
\author[Futorny,  Liu,   Lu,   Zhao]{Vyacheslav Futorny, Genqiang Liu, Rencai Lu, Kaiming Zhao}
\date{}\maketitle

\date{}\maketitle
\begin{abstract} Let $n>1$ be an integer, $\alpha\in{\mathbb C}^n$, $b\in{\mathbb C}$,  and  $V$ a $\mathfrak{gl}_n$-module. We define a class of weight modules $F^\alpha_{b}(V)$ over $\sl_{n+1}$ using the restriction of  modules of tensor fields over the Lie algebra of vector fields on $n$-dimensional torus.  In this paper we consider the case $n=2$ and  prove the irreducibility of  such 5-parameter $\mathfrak{sl}_{3}$-modules $F^\alpha_{b}(V)$  generically. All such modules have
 infinite dimensional weight spaces and lie outside of the category of Gelfand-Tsetlin modules. Hence, this construction yields new families of irreducible $\mathfrak{sl}_{3}$-modules.
 \end{abstract}


\vskip 5pt
\noindent
{\em 2010  Math. Subj. Class.:}
17B10, 17B20, 17B65, 17B66, 17B68

\vskip 10pt

\section{Introduction}
Representation theory of  infinite-dimensional Lie algebras has
been artacting extensive attentions of many mathematicians and
physicists. These Lie algebras include in particular the Witt algebra
$\mathcal{W}_n$ which is the derivation algebra of the Laurent
polynomial algebra $A_n=\C[x_1^{\pm1},x_2^{\pm1},\cdots,
x_n^{\pm1}]$. The algebra $\mathcal{W}_n$ is a natural higher rank
generalization of the Virasoro algebra, which has many applications
to different branches of mathematics and physics. It can also be defined as the Lie algebra of polynomial vector fields on
$n$-dimensional torus.
Weight representations of Witt algebras was recently
studied by many authors; see \cite{B, E1, E2, BMZ, GLZ,L3, L4, L5,
MZ2, Z, BF1, BF2, BF3}. In particular, in \cite{BF3}  a classification of weight irreducible $\mathcal{W}_n$-modules with finite weight multiplicities was given.

 In 1986,  Shen \cite{Sh} defined a class of modules $F^\alpha_b(V)$
over the Witt algebra $\mathcal{W}_n$ for  $\a\in\C^n$, $b\in\C$,
and an irreducible module $ V$ over the special linear Lie algebra
$\sl_n$.  This construction was also considered by  Larsson in  \cite{L3}.  
They are known as  \emph{modules of tensor fields}
when $V$ is finite dimensional and they have a geometric origin. These modules play essential part in the classification of 
weight irreducible $\mathcal{W}_n$-modules with finite weight multiplicities \cite{BF3}.
In 1996, Eswara Rao determined the necessary and
sufficient conditions for  modules of tensor fields to be irreducible  \cite{E1} (see also  \cite{GZ}). When $V$ is infinite dimensional,
the $\mathcal{W}_n$-module $F^\alpha_b(V)$ is always irreducible, see \cite{LZ}.
 With this paper we begin a systematic study
of the functors $$V\rightarrow F^\alpha_b(V),  \,\,   \ \  V\rightarrow F^\alpha_b(V)|_{\sl_{n+1}}$$ from the category of weight $\sl_n$-modules to the
the category of weight $\mathcal{W}_n$-modules and $\sl_{n+1}$-modules respectively.  Our goal is to construct irreducible
weight modules over $\sl_{n+1}$
with infinite dimensional weight spaces.

At the moment classification of all irreducible weight $\sl_{n+1}$-modules is only known for $n=1$. The largest subcategory of weight modules with infinite weight multiplicities
where
classification problem  can be handled  is the category of Gelfand-Tsetlin modules \cite{DFO}. A classification of irreducible modules in this category is complete for $n=2$ \cite{FGR1} and
is
known up to some finiteness  \cite{Ov}, \cite{FO} in general. Recently, new families of irreducible weight modules for $\sl_{n+1}$ were constructed by Futorny, Grantcharov and Ramirez
\cite{FGR2}, \cite{FGR3}, \cite{FGR4}.

It is not easy to construct examples of irreducible weight modules beyond the category of Gelfand-Tsetlin modules.  In the present paper  we treat the case $n=2$ and construct
new  families of 5-parameter
irreducible modules over the Witt algebra $\mathcal{W}_2$  with
infinite dimensional weight spaces and over the Lie algebra $\sl_{3}$. For the latter algebra, generic weight modules lie outside  of the category of Gelfand-Tsetlin modules.  We conjecture that
the same holds for an arbitrary $n\geq 2$.

\

\noindent{\bf Acknowledgements}. V.F. is supported in part by
the CNPq grant (301320/2013-6) and by the Fapesp grant (2014/09310-5); G.L. is partially supported by NSF of China (Grant
11301143) and  the school fund of Henan University (2012YBZR031, yqpy20140044); 
 K.Z. is supported in part by the Fapesp grant (2015/08615-0), by  NSF of China (Grant
11271109) and NSERC. He gratefully acknowledges the hospitality and excellent working conditions at
the S\~ao Paulo University where part of this work was done.

\section{$\sl_{n+1}$-modules from $\gl_n$-modules}
We denote by $\mathbb{Z}$, $\mathbb{Z}_+$, $\mathbb{N}$ and
$\mathbb{C}$ the sets of  all integers, nonnegative integers,
positive integers and complex numbers, respectively.

For a positive integer $n>1$, $\ZZ^{n}$ denotes the direct sum of $n$ copies of $\ZZ$.
 For any $a=(a_1,\cdots, a_n) \in \Z_+^n$ and $m=(m_1,\cdots,m_n)
 \in\C^n$, we denote $m^{a}=m_1^{a_1}m_2^{a_2}\cdots m_n^{a_n}$.  Let $\gl_n$ be the
Lie algebra of all $n \times n$ complex matrices, $\sl_n$ the
subalgebra of $\gl_n$ consisting of all traceless matrices. For $1
\leq i, j \leq n$ we denote by $E_{ij}$  the matrix units of    $\gl_n$   and by $\bE_{i,j}$  the   matrix units  of $\gl_{n+1}$.
 For a matrix $X$ we will denote by $X'$ its transpose.

\subsection{ Witt algebras $\mathcal{W}_n$}
We  denote by  $ \mathcal{W}_n$ the  Lie algebra of derivations of the
Laurent polynomial algebra $A_n=\C[t_1^{\pm1},t_2^{\pm1}, . . .
, t_n^{\pm1}]$, see \cite{L1}-\cite{L5}. For $i\in\{1,2,\dots,n\}$, denote $\partial
_i=t_i\frac{\partial}{\partial t_i}$; and for any
$a=(a_1,a_2,\cdots, a_n)\in\mathbb{Z}^n$ set  $t^a=t_1^{a_1}t_2^{a_2}\cdots t_n^{a_n}$.

Denote the standard basis by $\{e_1,e_2,...,e_n\}$  of the vector space $\mathbb{C}^n$.
For any $u\in \mathbb{C}^n$ we view $u'$ as a column vector.
Let
$(\,\cdot\,|\, \cdot\, )$ be the standard symmetric bilinear form
such that $(u | v)=uv'\in\mathbb{C}$.  For $u \in \mathbb {C}^n$
and $r\in \mathbb{Z}^n$, we set
$D(u,r)=t^r\sum_{i=1}^nu_i\partial_i$. Then we have
$$[D(u,r),D(v,s)]=D(w,r+s),\,u,v\in \mathbb {C}^n, r,s\in \mathbb {Z}^n,$$
where $w=(u | s)v-(v | r)u$. Note that for any $u,v,z,y\in
\mathbb{C}^n$, both $u'v $ and $x'y $ are $n\times n$ matrices, and
 \begin{equation*}(u'v)(z'y )=(v|z)u'y .\end{equation*}
A subalgebra $\h=\span\{\partial_1, \partial_2, ... , \partial_n\}$ is the Cartan
 subalgebra of $\mathcal{W}_n$. 

 The extended Witt algebra
$\widetilde{\mathcal{W}}_n$ is the semidirect sum of $\W_n$ and the
abelian Lie algebra $A_n$ with intertwining brackets
$$[D(u,r), t^s]=(u|s)t^{r+s}, \,\,\forall\,\,\,u\in \mathbb {C}^n, r,s\in \mathbb {Z}^n.$$

For  any  $\alpha\in \mathbb{C}^n, b\in\C$ and a $\gl_n$-module $V$ on
which the  identity matrix acts as the scalar $b$, set
$F^\alpha_b(V)=V\otimes A_n$. For simplicity we write $v(n) = v
\otimes x^n$ for any $v\in V, n\in\Z^n$.
 Then $F^\alpha_b(V)$  becomes a
${\mathcal{W}}_n$-module with respect to  the following
action
 \begin{equation}D(u,r)v(n)=\Big((u \mid n+\alpha)v+
 (r'u)v\Big)(n+r),\end{equation}
where $u\in\mathbb{C}^n$, $v\in V$,  $ n, r\in\mathbb{Z}^n$ see \cite{L1} and \cite{Sh}. It is
easy to see that the module $F^\alpha_b(V)$ obtained from any $\gl_n$-module $V$ is
always a weight module over $\mathcal{W}_n$. The following result is well-known \cite{E1, Ru, GZ, Z}.


\begin{theorem}\label{t} Let $\alpha\in \mathbb{C}^n, b\in\C$, and let $V$ be
 an irreducible finite dimensional module over $\gl_n$ on  which the identity matrix acts as the scalar
$b$. Then $F^\alpha_b(V)$ is irreducible
${\mathcal{W}}_n$-module unless it appears in the de Rham
complex of differential forms
$$t^{\alpha} \Omega^0 \rightarrow t^{\alpha} \Omega^1 \rightarrow \; \ldots \; \rightarrow t^{\alpha} \Omega^n.$$
The middle terms in this complex are reducible $W_n$-modules,
while the modules
$t^{\alpha} \Omega^0$ and $t^{\alpha}
\Omega^n$ are reducible whenever $\alpha\in\Z^n$. Here $t^{\alpha} \Omega^i\simeq F^\alpha_b(\Lambda^i(\mathbb{C}^n)).$
\end{theorem}

\subsection{ Defining $\sl_{n+1}$-modules}
Recall a standard embedding of $\sl_{n+1}$ into ${\mathcal{W}}_n$. Set  $$\begin{aligned}&\bE_{ij}=t_it_j^{-1}\partial_j, 1\leq i,j\leq n;\\
&\bE_{i,n+1}=-t_i\sum_{j=1}^n\partial_j,  \  \  \  \bE_{n+1,i}=t_i^{-1}\partial_i,
1\leq i\leq n;\\ &\bE_{n+1,n+1}=-\sum_{j=1}^n\partial_j.\end{aligned}$$
Then it is well-known that the linear span over $\C$ of the set
$$\{\bE_{ij}:1\le i\ne j\le  n+1\}\cup \{\bE_{ii}-\bE_{i+1,i+1}: 1\le i\le
n\}$$
 is isomorphic to the Lie algebra $\mathfrak{sl}_{n+1}(\C)$, see for example \cite{M}. So
 $\mathfrak{sl}_{n+1}(\C)$ can be regarded as a subalgebra of
$ {\mathcal{W}}_n$ and each
$ {\mathcal{W}}_n$-module can be seen as a
$\mathfrak{sl}_{n+1}(\C)$-module by restriction. We fix the following {\emph {Cartan subalgebra}} of
$\mathfrak{sl}_{n+1}(\C)$:
$${\mathfrak{h}}=\bigoplus_{i=1}^{n-1}\C(\partial_i-\partial_{i+1})+\C(\sum_{j=1}^n\partial_j+\partial_n)
=\bigoplus_{i=1}^n\C\partial_i.$$  Also, set
$$\mathfrak{n}_{+} =\bigoplus_{1\le i<j\le n+1}\C \bE_{ij} , \ \ \ \mathfrak{n}_{-}=\bigoplus_{1\le j<i\le n+1}\C \bE_{ij}$$
Then $\mathfrak{sl}_{n+1}(\C)$ has a triangular decomposition
$\mathfrak{sl}_{n+1}(\C)=\mathfrak{n}_{-}\oplus
{\mathfrak{h}}\oplus\mathfrak{n}_{+}.$
Restricting a  $ {\mathcal{W}}_n$-module
$F^\alpha_b(V)$ onto $\sl_{n+1}$ we obtain an $\sl_{n+1}$-module which we call \emph{Witt module} and denote again as $F^\alpha_b(V)$.
 It is easy to see that all $F^\alpha_b(V)$ are weight (with respect to ${\mathfrak{h}}$) modules over
$\sl_{n+1}$.

 More precisely, we have
\begin{equation}\aligned &\bE_{ij}\cdot v(r)=((r_j+\a_j)v+(E_{ij}-E_{jj})v)(r+e_i-e_j);\\
&\bE_{i,n+1}\cdot v(r)=-(\sum_{j=1}^n (\a_j+r_j)v+\sum_{j=1}^n E_{ij}v)(r+e_i);\\
&\bE_{n+1,i}\cdot v(r)=((r_i+\a_i)v-E_{ii}v)(r-e_i);\\
&\bE_{n+1,n+1}\cdot v(r)=-\sum_{j=1}^n (\a_j+r_j)v(r),\endaligned\end{equation}
for all $v\in V$ and $1\le i,j\le n$.

\section{ From $\gl_2$-modules to $\sl_3$-modules}

In this section we consider Witt modules over $\sl_3$ obtained from infinite dimensional weight $\gl_2$-modules.
Let $\alpha_1,\alpha_2,\lambda, b, c \in\C$ with $c\pm \lambda\notin\Z$.
Let $V=\span\{v_i|i\in\Z\}$ be an irreducible cuspidal (i.e., with injective action of generators $E_{12}$, $E_{21}$) $\gl_2$-module with $1$-dimensional weight subspaces and
 with the following  module structure:
\begin{equation}\label{V}\aligned &E_{11}v_i=(b+i'')v_i,\\
&E_{22}v_i=(b-i'')v_i,\\
&E_{12}v_i=(c+i'')v_{i+1},\\
&E_{21}v_i=(c-i'')v_{i-1},\endaligned\end{equation}
where $i''=\lambda+i$.
It is well known that these modules exhaust all irreducible cuspidal weight  $\gl_2$-modules.
For convenience, we set
$$r_i'=r_i+\alpha_i,\forall r_i\in\Z, i=1,2.$$
We will consider the $\sl_3$-module $$F^{\alpha}_{2b}(V)=\span\{v_i(r_1, r_2)|i, r_1,r_2\in\Z\}.$$
Then we have
\begin{equation}\aligned &\bE_{11}v_i(r_1, r_2)\hskip -4pt =\hskip -4pt r'_1v_i(r_1, r_2),\\
&\bE_{22}v_i(r_1, r_2)\hskip -4pt = \hskip -4pt r'_2v_i(r_1, r_2),\\
&\bE_{33}v_i(r_1, r_2)\hskip -4pt =\hskip -4pt 
-( r_1' +r_2')v_i(r_1, r_2),\\
&\bE_{12}v_i(r_1, r_2)
\hskip -4pt =\hskip -4pt \Big((i''-b+r_2' )v_{i} +(c +i'')v_{i+1}\Big)(r_1+1, r_2-1),\\
&\bE_{21}v_i(r_1, r_2)
\hskip -4pt =\hskip -4pt \Big((c -i'')v_{i-1}+(-i''-b +r_1' )v_{i}\Big)(r_1-1, r_2+1),\\
&\bE_{13}v_i(r_1, r_2)
\hskip -4pt =\hskip -4pt -\Big((r'_1 +r_2'+b +i'')v_i+(c +i'')v_{i+1}\Big)(r_1+1, r_2),\\
&\bE_{31}v_i(r_1, r_2)
\hskip -4pt =\hskip -4pt (r_1'-b-i'')v_i(r_1-1, r_2),\\
&\bE_{23}v_i(r_1, r_2)
\hskip -4pt =\hskip -4pt -\Big((c-i'')v_{i-1}+(r_1'+r_2'+b-i'')v_i\Big)(r_1, r_2+1),\\
&\bE_{32}v_i(r_1, r_2)
=(r_2'-b+i'')v_i(r_1, r_2-1)
.\endaligned\end{equation}

\begin{lemma}\label{2.2} Let $\lambda, b, c,  \alpha_1,\alpha_2\in\C$ such that $c\pm \lambda, \alpha_1-b- \lambda,
\alpha_2-b+\lambda,\alpha_1+2b,  \alpha_2+2b,
\alpha_1 +\alpha_2+b\pm c\notin \Z$, and $V$ be the $\gl_2$-module defined by \eqref{V}.
 Then the $\sl_3$-module $F^\alpha_{2b}(V)$ can be generated by any single vector $v_i(r_1,r_2)$ where $i, r_1,r_2\in\Z$.
\end{lemma}

\begin{proof}Let $W$ be the $\sl_3$-submodule of $F^\alpha_{2b}(V)$ generated by  $v_i(r_1,r_2)$ where  $i, r_1,r_2\in\Z$ are fixed.

\

{\bf Claim.} We have $v_i(j, r_1+r_2-j)\in W$ for all $j\in\Z$.

Indeed, from
$$\aligned &\bE_{12}v_i(r_1, r_2)=\Big((i''-b+r_2' )v_{i} +(c+i'')v_{i+1}\Big)(r_1+1, r_2-1),\\
&\bE_{13}\bE_{32}v_i(r_1, r_2)=(r_2'- b+i'')\bE_{13}(v_i(r_1, r_2-1)),\\
&\hskip 1cm  =-(r_2'-b  +i'')\Big((r_1' +r_2'-1+b+i'')v_i \\
&\hskip 1.5cm+(c+i'')v_{i+1}\Big)(r_1+1, r_2-1),
\endaligned
$$
and noting that $ \alpha_1+2b, \alpha_2-b+\lambda\notin\Z$ we see that
\begin{equation}\label{vi1}v_i(r_1+1, r_2-1), v_{i+1}(r_1+1,r_2-1)\in W.\end{equation}
 Now, from
$$\aligned
\bE_{21}v_i(r_1, r_2) =& \Big((c -i'')v_{i-1}\\
&+(-i''-b +r_1')v_{i}\Big)(r_1-1, r_2+1),\\
\bE_{23}\bE_{31}v_i(r_1, r_2)=&(r_1' -b -i'')\bE_{23}(v_i(r_1-1, r_2))\\
 \ \ \ \ =& -(r_1' -b -i'')\Big((c-i'')v_{i-1}\\
 \hskip 2cm & +( r_1' +r'_2+b -i''-1)v_i\Big)(r_1-1, r_2+1)\endaligned
$$ and using the fact  that $ \alpha_2+2b, \alpha_1-b- \lambda\notin\Z$  we obtain
\begin{equation}\label{vi2}v_i(r_1-1, r_2+1), v_{i-1}(r_1-1,r_2+1)\in W.\end{equation}
In this manner we deduce the claim.

For integer   $s, r\in\Z$ set $V(s, r)=\span\{v(s, r)|v\in V\}.$

Repeatedly using Claim,  \eqref{vi1} and \eqref{vi2}, we deduce that
$$V(j, r_1+r_2-j)\subset W, \forall j\in\Z.$$
Applying $\bE_{31}$ we obtain that
$$V(j_1,j_2)\subset W, \forall j_1,j_2\in\Z, j_1+j_2\le r_1+r_2.$$

Since 
$$\aligned
\bE_{13}v_i(j_1, j_2)=&-\Big(( j'_1 +j'_2+b+i'')v_i\\
\hskip 2cm& +(c+i'')v_{i+1}\Big)(j_1+1, j_2),\\
\bE_{23}v_{i+1}(j_1+1, j_2-1) =& -\Big((c-i''-1)v_{i}\\
\hskip 2cm & +( j_1' + j'_2+b-i''-1)v_{i+1}\Big)(j_1+1, j_2),\endaligned$$
 and $\alpha_1 +\alpha_2+b\pm c\notin \Z$, we see that
\begin{equation}v_i(j_1+1, j_2), v_{i+1}(j_1+1,j_2)\in W.\end{equation}
Repeatedly using the claim  and the above equation we deduce that $W=F^\alpha_{2b}(V)$, and hence the lemma follows.
\end{proof}

\begin{theorem}\label{2.3} Let $\lambda, b, c, \alpha_1,\alpha_2\in\C$ such that $ c\pm \lambda, \alpha_1-b- \lambda,
\alpha_2-b+\lambda,\alpha_1+2b,  \alpha_2+2b,
\alpha_1 +\alpha_2+b\pm c$,  and $c\pm 3b$ are not integers. If $V$ is a  $\gl_2$-module define by \eqref{V}
 then the $\sl_3$-module $F^\alpha_{2b}(V)$ is irreducible.
\end{theorem}

\begin{proof} To the contrary, we assume that  the $\sl_3$-module  $F^\alpha_{2b}(V)$ is reducible.   Take a  nonzero proper submodule $W$ of $F^\alpha_{2b}(V)$.
For \begin{equation}\label{wi}w_{i}(r_1,r_2)=\sum_{j=0}^sa_{ij}(r_1,r_2)v_{i+j}(r_1,r_2)\in F^\alpha_{2b}(V)\end{equation}
where $a_{ij}(r_1,r_2)\in\C$ with $a_{i0}(r_1,r_2)=1, a_{is}(r_1,r_2)\ne0$, we say that the length of $w_{i}(r_1,r_2)$ is $s+1$.
We may assume that $s+1$ is the minimal length of nonzero homogeneous elements in $W$, in particular, $w_{i}(r_1,r_2) \in W$ for some fixed $i, r_1,r_2\in\Z$.
We know that $s\ge 1$.

We will first  prove that all vectors in (\ref{wi}) form a basis of $W$. Then analyze the coefficients $a_{ij}(r_1,r_2)$ to deduce some contradictions.

\

{\bf Claim:} The submodule  $W$ is spanned by elements of length $s+1$, i.e.,   $W$ is spanned by all $w_{j}(s_1,s_2)$ for $j, s_1,s_2\in\Z$, which are defined as in (\ref{wi}).

To prove this claim, we essentially follow the proof of Lemma \ref{2.2}.  Since $\alpha_1 +2b, \alpha_2-b+\lambda\notin\Z$, the following two elements in $W$:
$$\aligned \bE_{12}w_i(r_1, r_2)=&\sum_{j=0}^sa_{ij}(r_1,r_2)(i''+j-b+r_2')v_{i+j}(r_1+1, r_2-1) \\
& \ \ \ +\sum_{j=0}^sa_{ij}(r_1,r_2)(c+i''+j)v_{i+j+1}(r_1+1, r_2-1),\\
 \bE_{13}\bE_{32}w_i(r_1, r_2)&=\sum_{j=0}^sa_{ij}(r_1,r_2)(r_2'-b+i''+j)\bE_{13}(v_{i+j}(r_1, r_2-1))\\
=-\sum_{j=0}^sa_{ij}&(r_1,r_2)(r_2'-b+i''+j ) \Big(( r'_1 +r'_2-1+b+i''+j )v_{i+j}\\
& \ \ \ \ \ +(c+i''+j)v_{i+j+1}\Big)(r_1+1, r_2-1),
\endaligned
$$
are linearly independent which can be seen by looking at the coefficients of $v_i(r_1, r_2)$ and $v_{i+s+1}(r_1, r_2)$. By taking   linear combinations of the above two elements we obtain  two homogeneous elements of the form (\ref{wi}):
\begin{equation}\label{vi1'}w_i(r_1+1, r_2-1), w_{i+1}(r_1+1,r_2-1)\in W.\end{equation}

Also, since $\alpha_2 +2b, \alpha_1-b- \lambda\notin\Z$, the following two elements in $W$:
$$\aligned
\bE_{21}w_i(r_1, r_2) =&\sum_{j=0}^sa_{ij}(r_1,r_2)\Big((c-i''-j)v_{i+j-1}\\
& \  \  \  \ +(-i''-j-b+r_1')v_{i+j}\Big)(r_1-1, r_2+1),\\
\bE_{23}\bE_{31}w_i(r_1, r_2)
& =-\sum_{j=0}^sa_{ij}(r_1,r_2)(r_1'-b-i''-j)\Big((c-i''-j)v_{i+j-1}\\
& \ \ \ +(r_1' +r'_2+b-i''-j-1)v_{i+j}\Big)(r_1-1, r_2+1)\endaligned
$$
are linearly independent which can be seen by looking at the coefficients of $v_{i-1}(r_1, r_2)$ and $v_{i+s}(r_1, r_2)$. By taking   linear combinations of the above two elements we obtain  two homogeneous elements of the form (\ref{wi}): \begin{equation}\label{vi2'}w_i(r_1-1, r_2+1), w_{i-1}(r_1-1,r_2+1)\in W.\end{equation}
In this manner we obtain vectors $$w_{k}(j, r_1+r_2-j)\in W, \forall\ k, j\in\Z.$$
Applying $\bE_{31}$, we deduce that
$$w_k(j_1,j_2)\in W, \forall \ k, j_1,j_2\in\Z, j_1+j_2\le r_1+r_2.$$

From $$\aligned
&\bE_{13}w_k(j_1, j_2)=-\sum_{j=0}^sa_{kj}(j_1,j_2)\Big((j_1'+j'_2+b+k''+j)v_{k+j}\\
 &\hskip 1.5cm+(c+k'')v_{k+j+1}\Big)(j_1+1, j_2)\in W,\\
&\bE_{23}w_{k+1}(j_1+1, j_2-1) =-\sum_{j=0}^sa_{kj}(j_1,j_2)\Big((c-k''-1)v_{k+j}\\
&\hskip 1.5cm+(j'_1 +j'_2+b-k''-j-1)v_{k+j+1}\Big)(j_1+1, j_2)\in W,\endaligned$$
 and the fact that $\alpha_1 +\alpha_2+b\pm c\notin \Z$, we obtain
\begin{equation}w_k(j_1+1, j_2), w_{k+1}(j_1+1,j_2)\in W,  \forall \ k, j_1,j_2\in\Z, j_1+j_2\le r_1+r_2.\end{equation}
Repeating this precess, we deduce the Claim.

\

Now we have $$\bE_{31}w_{i}(r_1,r_2)=\sum_{j=0}^s(r_1'-b -i''-j)a_{ij}(r_1,r_2)v_{i+j}(r_1-1,r_2).$$ Hence,
\begin{equation}\label{ar1} a_{ij}(r_1-1,r_2)=\frac{r_1'-b- i''-j}{r_1'-b-i''} a_{ij}(r_1,r_2) , \end{equation}
for  $i, r_1,r_2\in\Z$ and $ j=1,2,\cdots, s$.
Since $$\bE_{32}w_{i}(r_1,r_2)=\sum_{j=0}^s(r_2'-b+i''+j)a_{ij}(r_1,r_2)v_{i+j}(r_1,r_2-1),$$
we also see that
\begin{equation}\label{ar2} a_{ij}(r_1,r_2-1)=\frac{r_2'-b+ i''+j}{r_2'-b+ i''} a_{ij}(r_1,r_2) , \end{equation}
for  $i, r_1,r_2\in\Z$ and $i, r_1,r_2\in\Z$ and $ j=1,2,\cdots, s$.

Cancelling the term $v_{i+s+1}(r_1+1,r_2-1)$  in the formulas of $\bE_{12}w_{i}(r_1,r_2)$ and $\bE_{13}\bE_{32}w_{i}(r_1,r_2)$, we have
\begin{equation}\label{ar12} \aligned&\Big((r_2'-b+i''+s)\bE_{12}+\bE_{13}\bE_{32}\Big)w_{i}(r_1,r_2)\\
&= (r_2'-b+i'')(s+1-2b-r_1')w_i(r_1+1,r_2-1), \endaligned\end{equation}
i.e.,
$$\aligned
&(r_2'-b+i''+s)\Big((r_2'-b +i''+j)a_{i,j}(r_1,r_2)\\
& \  \  \  \  \  +(c +i''+j-1)a_{i,j-1}(r_1,r_2)\Big)\\
&-(r_2'-b+ i''+j)(r_1'+r_2'-1+b +i''+j)a_{i,j}(r_1,r_2)\\
&-(r_2'-b+i''+j-1)(c+i''+j-1)a_{i,j-1}(r_1,r_2)\\
&= (r_2'-b +i'')(s+1-2b-r_1')a_{ij}(r_1+1,r_2-1),\\
&=(s+1-r_1'-2b)\frac{(r_1'-b-i''+1)(r_2'-b +i''+j)}{(r_1'-b-i''-j+1)}a_{ij}(r_1,r_2),
\endaligned$$ where in the last step we have used \eqref{ar1} and \eqref{ar2}.
So we have
$$\aligned
&(r_2'-b+i''+j) (s+1-r_1'-2b-j)a_{i,j}(r_1,r_2)\\
&+(c+i''+j-1)(s+1-j)a_{i,j-1}(r_1,r_2)\\
&=\frac{
(r_2' -b+i''+j) (s+1-r_1'-2b)(r_1'-b -i''+1) }{(r_1'-b -i''-j+1)}a_{ij}(r_1,r_2),
\endaligned$$ i.e.,
$$\aligned
&(s+1-j)(c +i''+j-1)a_{i,j-1}(r_1,r_2)\\
&=\frac{(r_2'-b +i''+j) j(s+2-3b -i''-j) }{(r_1'-b- i''-j+1)}a_{ij}(r_1,r_2).
\endaligned$$ or
\begin{equation} \label{jj-1}a_{i,j-1}(r_1,r_2)
=\frac{
(r_2'-b+i''+j) j(s+2-3b-i''-j) }{(r_1'-b -i''-j+1)(s+1-j)(c +i''+j-1)}a_{ij}(r_1,r_2),
\end{equation}
for $i, r_1,r_2\in\Z$ and $ j=1,2,\cdots, s$.
From this formula we see that $\lambda+3b\notin\Z$.

Cancelling the term $v_{i-1}(r_1-1,r_2+1)$ in  the formulas of  $\bE_{21}w_{i}(r_1,r_2)$ and $\bE_{23}\bE_{31}w_{i}(r_1,r_2)$, we see that
\begin{equation}\label{ar12'} \aligned&\Big((r_1'-b -i'')\bE_{21}+\bE_{23}\bE_{31}\Big)w_{i}(r_1,r_2)\\
&= \frac{(r_1'-b -i''-s)(s+1-2b-r_2')a_{is}(r_1,r_2)}{a_{is}(r_1-1,r_2+1)}w_i(r_1-1,r_2+1), \endaligned\end{equation}
i.e.,
$$\aligned
&(r_1'-b -i'')\Big((-i''-j-b +r_1')a_{i,j}(r_1,r_2)\\
& \  \  \  \  \  +(c -i''-j-1)a_{i,j+1}(r_1,r_2)\Big)\\
&-(r_1'-b -i''-j)( r'_1-1 +r'_2+b -i''-j)a_{i,j}(r_1,r_2)\\
&-(r_1'-b -i''-j-1)(c -i''-j-1) a_{i,j+1}(r_1,r_2)\\
&= \frac{(r_1'-b -i-s)(s+1-2''b-r_2')a_{is}(r_1,r_2)}{a_{is}(r_1-1,r_2+1)}a_{ij}(r_1-1,r_2+1),\\
&=\frac{(s+1-r_2'-2b )(r_2'-b +i''+s+1)(r_1'-b -i''-j)}{(r_2'-b +i''+j+1)}a_{ij}(r_1,r_2),
\endaligned$$ where in the last step we have used \eqref{ar1} and \eqref{ar2}.
Hence,
$$\aligned
&(r_1'-b -i''-j)(j+1-r_2'-2b )a_{i,j}(r_1,r_2)\\
&+(c -i''-j-1)(j+1)a_{i,j+1}(r_1,r_2)\\
&=\frac{(s+1-r'_2-2b )(r_2'-b+i''+s+1)(r_1'-b -i''-j)}{(r_2'-b +i''+j+1)}a_{ij}(r_1,r_2).
\endaligned$$ i.e.,
$$\aligned
&(c- i''-j-1)(j+1)a_{i,j+1}(r_1,r_2)\\
&=\frac{
(r_1'+-b-i''-j)(s- j)(s+2-3b+  i''+j) }{r_2'-b+ i''+j+1}a_{ij}(r_1,r_2).
\endaligned$$ or
$$\aligned
&(c -i''-j)ja_{i,j}(r_1,r_2)\\
&=\frac{
(r_1'-b -i''-j+1)(s+1- j)(s+1-3b +i''+j) }{r_2'-b +i''+j}a_{i,j-1}(r_1,r_2),
\endaligned$$
for $i, r_1,r_2\in\Z$ and $ j=1,2,\cdots, s$. From this formula we see that $\lambda-3b\notin\Z$.
So we have
$$\aligned
&a_{i,j-1}(r_1,r_2)=\\
&\frac{(c- i''-j)j(r_2'-b +i''+j)}{
(r_1'-b-i''-j+1)(s+1- j)(s+1-3b+  i''+j) }a_{i,j}(r_1,r_2),
\endaligned$$
for $i, r_1,r_2\in\Z$ and $ j=1,2,\cdots, s$. Compare this with \eqref{jj-1} we see that
$$\aligned&\frac{(c-i''-j)j(r_2'-b+ i''+j)}{
(r_1'-b-i''-j+1)(s+1- j)(s+1-3b+i''+j) }\\
=&\frac{
(r_2'-b+i''+j) j(s+2-3b- i''-j) }{(r_1'-b-i''-j+1)(s+1-j)(c+i''+j-1)},\endaligned
$$ i.e.,
$$\frac{c-i''-j }{
 s+1-3b+ i''+j}\\
=\frac{
 s+2-3b- i''-j }{ c+i''+j-1}, \forall i\in\Z.
$$
It implies
$$c+3b-s-2=0 {\text{  or  }} c-3b+s+3=0.
$$
This is impossible since we have assumed that $c\pm 3b\notin\Z$. Thus the $\sl_3$-module $F^\alpha_{2b}(V)$ is irreducible.\end{proof}

Theorem \ref{2.3}  provides a large family of weight modules over $\sl_3$ with infinite dimensional weight spaces with explicit basis and the action of
 generators of the Lie algebra. We will call a Witt module $F^\alpha_{2b}(V)$ \emph{generic} if its parameters satisfy conditions of  Theorem \ref{2.3}. For example, it is the case if the five parameters $\lambda, b, c, \alpha_1,\alpha_2\in\C$  and $1$ are  linearly independent over $\Q$.

\begin{remark}
  Assume $\a_1-b-\lambda=k_1\in\Z$  and $\a_2-b+\lambda=k_2 \in \Z$.   Then
$X=\{v\in F_{2b}^{\alpha}(V)\ | \  E_{31}v=E_{32}v=0\}\ne 0$ and it generates  a proper $\sl_3$-submodule of $F_{2b}^{\alpha}(V)$. This submodule is
 a  generalized Verma module over $\sl_3$.  This shows that generalized Verma modules can be recovered from Witt modules.
\end{remark}

\section{Witt modules versus Gelfand-Tsetlin modules}

In this section we show that irreducible generic Witt modules are not Gelfand-Tsetlin modules.
 Let us first recall Gelfand-Tsetlin modules for $\gl_n$ which were introduced in \cite{DFO}.
Let  $U=U(\gl_n)$.  For each $m\leqslant n$ let $\mathfrak{gl}_{m}$ be the Lie subalgebra
of $\gl_n$ spanned by $\{ E_{ij}\,|\, i,j=1,\ldots,m \}$. Then we have the following chain
$$\gl_1\subset \gl_2\subset \ldots \subset \gl_n,$$
which induces  the chain $U_1\subset$ $U_2\subset$ $\ldots$ $\subset
U_n$ of the universal enveloping algebras  $U_{m}=U(\gl_{m})$, $1\leq m\leq n$. Let
$Z_{m}$ be the center of $U_{m}$. Then $Z_m$ is the polynomial
algebra in the $m$ variables $\{ c_{mk}\,|\,k=1,\ldots,m \}$,
\begin{equation}\label{equ_3}
c_{mk } \ = \ \displaystyle {\sum_{(i_1,\ldots,i_k)\in \{
1,\ldots,m \}^k}} E_{i_1 i_2}E_{i_2 i_3}\ldots E_{i_k i_1}.
\end{equation}

 The subalgebra $\Gamma$ of $U$ generated by $\{
Z_m\,|\,m=1,\ldots, n \}$ is the \emph{Gelfand-Tsetlin
subalgebra} of $U$ associated with the chain of subalgebras above. It is a polynomial algebra in the $\displaystyle \frac{n(n+1)}{2}$ variables $\{
c_{ij}\,|\, 1\leqslant j\leqslant i\leqslant n \}$.

A finitely generated $U$-module
$M$ is called a \emph{Gelfand-Tsetlin module (with respect to
$\Gamma$)} if as a $\Gamma$-module

\begin{equation*}
M=\bigoplus_{\mathfrak m\in  \max(\Gamma)}M(\mathfrak m),
\end{equation*}
where $$M(\mathfrak m)=\{v\in M| \mathfrak m^{k}v=0 \text{ for some }k\geq 0\},$$ and $ \max(\Gamma)$ is the set of maximal ideals of $\Gamma$.

Since the Gelfand-Tsetlin subalgebra contains the Cartan subalgebra $\h$ spanned by  $\{ E_{ii}\,|\, i=1,\ldots, n\}$ then any irreducible Gelfand-Tsetlin module is a weight module.
On the other hand, every weight module with finite dimensional weight subspaces is a Gelfand-Tsetlin module. Theory of Gelfand-Tsetlin modules was developed in \cite{Ov}, \cite{FO}.
In fact,
there are
several
 maximal commutative subalgebras of Gelfand-Tsetlin type for which we can define a corresponding category of Gelfand-Tsetlin modules. These subcategories correspond to different (finitely many) chains of different embedding as above. In the case of $\sl_3$ we will have $3$ possible Gelfand-Tsetlin subalgebras. Since different chains
  are
 conjugated by the Weyl group, respective categories of Gelfand-Tsetlin  modules are equivalent.

  Denote by $\mathcal GT$ the category of modules which are
  Gelfand-Tsetlin  modules with respect to some  Gelfand-Tsetlin subalgebra.

  \begin{proposition}
 If $F^\alpha_{2b}(V)$ is a generic Witt module, then the $\sl_3$-module $F^\alpha_{2b}(V)$ is not a Gelfand-Tsetlin  module.

 \end{proposition}

 \begin{proof}
 Suppose $F^\alpha_{2b}(V)$ is irreducible and $F^\alpha_{2b}(V)\in  \mathcal GT$. Then $F^\alpha_{2b}(V)$  contains a nonzero element $v$ which is  a common eigenvector of one of three Gelfand-Tsetlin subalgebras.
   This is possible if and only if $v$ is an eigenvector for one of the following operators: $\bE_{12}\bE_{21}$, $\bE_{23}\bE_{32}$, $\bE_{13}\bE_{31}$. From (3.2) we have
   $$\bE_{12}\bE_{21}v_i(r_1, r_2)$$
   $$=(c-i'')\bE_{12}v_{i-1}(r_1-1, r_2+1)+(-i''-b +r_1' )\bE_{12}v_{i}(r_1-1, r_2+1)$$
   $$=\ldots + (-i''-b +r_1' )(c+ i'')v_{i+1}(r_1, r_2),$$
   where $\ldots$ corresponds to the terms with smaller indices. Since $ (-i''-b +r_1' )(c+ i'')\neq 0$ 
   for all integer $i$, then $v_i(r_1, r_2)$ is not an eigenvector
   of $\bE_{12}\bE_{21}$. We immediately conclude that  $\bE_{12}\bE_{21}$ has no
   eigenvectors in $F^\alpha_{2b}(V)$.
   Similarly,
   $$\bE_{23}\bE_{32}v_i(r_1, r_2)=-(r_2'-b+ i'')(c- i'')v_{i-1}(r_1, r_2)+  \ldots,$$
    where $\ldots$ corresponds to the terms with larger indices. Then $\bE_{23}\bE_{32}$ has no
   eigenvectors in $F^\alpha_{2b}(V)$ since $(r_2'-b+i'')(c- i'')\neq 0$ for all integer $i$.

   Consider now
   $$\bE_{13}\bE_{31}v_i(r_1, r_2)=-(r_1'-b- i'')(c+ i'')v_{i+1}(r_1, r_2)+ \ldots,$$
   where $\ldots$ corresponds to the terms with smaller indices.  Hence, $\bE_{23}\bE_{32}$ has no
   eigenvectors in $F^\alpha_{2b}(V)$ since  $(r_1'-b- i'')(c+\lambda+i)\neq 0$ for all integer $i$. Thus any generic Witt module   $F^\alpha_{2b}(V)$ is not a Gelfand-Tsetlin  module.
 \end{proof}

 In \cite{FOS} torsion theories for $\gl_n$ were studied and a stratification of the category of all weight modules by the heights of prime ideals was obtained.
 In this stratification the category of Gelfand-Tsetlin modules is a starting point (when prime ideals are maximal).  We believe that  Witt modules will provide examples of irreducible modules in other strata as in the case of $\sl_3$.

\vspace{3mm}

 \noindent  V.F.: Institute of Mathematics and Statistics, University of Sao Paulo, Brazil, 05315-970.
 Email: futorny@ime.usp.br

 \vspace{0.2cm}  \noindent G.L.: Department of Mathematics, Henan University, Kaifeng 475004, China. Email:
liugenqiang@amss.ac.cn

\vspace{0.2cm} \noindent R,L.: Department of Mathematics, Soochow
University,  Suzhou,  P. R. China, Email: rencail@amss.ac.cn

\vspace{0.2cm} \noindent K.Z.: Department of Mathematics, Wilfrid
Laurier University, Waterloo, ON, Canada N2L 3C5,  and College of
Mathematics and Information Science, Hebei Normal (Teachers)
University, Shijiazhuang, Hebei, 050016 P. R. China. Email:
kzhao@wlu.ca


\begin{thebibliography}{99999}

\bibitem[B]{B} Y. Billig, Jet modules, Canad. J. Math., 59 (2007), no.4, 712-729.
\bibitem[BB]{BB} S. Berman and Y. Billig, Irreducible representations for toroidal Lie algebra,
J. Algebra, 221(1999), 188-231.
\bibitem[BF1]{BF1}  Billig Y., Futorny, V. Representations of Lie algebra of vector fields on a torus
and chiral de Rham complex,Transactions of the American Mathematical
Society 366 (2014), 4697-4731.
\bibitem[BF2]{BF2} Y. Billig and V. Futorny, Billig, Y., Classification of simple cuspidal modules for solenoidal Lie algebras, Israel Journal of Mathematics, to appear.
\bibitem[BF3]{BF3} Y. Billig and V. Futorny, Classification of simple $W_n$-modules with
finite-dimensional weight spaces, Journal fur die Reine und Ange- wandte Mathematik, 720 (2016), 199-216.
\bibitem[BMZ]{BMZ} Y. Billig, A. Molev, R. Zhang, Differential equations in vertex algebras and
simple modules for the Lie algebra of vector fields on a torus, Adv.
Math., 218(2008), no.6, 1972-2004.
\bibitem[BZ]{BZ} Y. Billig and K. Zhao, Weight modules over exp-polynomial Lie algebras, J. Pure Appl.
Algebra, 191(2004), 23-42
\bibitem[DMP]{DMP}I. Dimitrov, O. Mathieu, I. Penkov, On the structure of weight modules.
Trans. Amer. Math. Soc. 352 (2000), no. 6, 2857-2869.
\bibitem[DFO]{DFO} Y. Drozd, S. Ovsienko, V. Futorny,  Harish-Chandra subalgebras and Gelfand-Zetlin modules, {Math. and Phys. Sci.} {424} (1994), 72--89.
\bibitem[E1]{E1} S. Eswara Rao, Irreducible representations of
the Lie-algebra of the diffeomorphisms of a $d$-dimensional torus,
J. Algebra, {\bf 182}  (1996),  no. 2, 401--421.
\bibitem[E2]{E2} S. Eswara Rao, Partial classification of modules for Lie algebra of
diffeomorphisms of d-dimensional torus, J. Math. Phys., 45 (8),
(2004) 3322-3333.
\bibitem[FGR1]{FGR1} V. Futorny, D. Grantcharov, L. E. Ramirez, Classification of irreducible Gelfand-Tsetlin modules for $\mathfrak{sl}(3)$. In progress.
\bibitem[FGR2]{FGR2} V. Futorny, D. Grantcharov, L.E. Ramirez, Irreducible Generic Gelfand-Tsetlin Modules of $\gl(n)$, Symmetry, Integrability and Geometry: Methods and Applications, v. 018, (2015).
\bibitem[FGR3]{FGR3} V. Futorny, D. Grantcharov, L. E. Ramirez, Singular Gelfand-Tsetlin modules for $\mathfrak{gl}(n)$. Adv. Math., 290, (2016), 453--482.
\bibitem[FGR4]{FGR4} V. Futorny, D. Grantcharov, L. E. Ramirez, New singular Gelfand-Tsetlin $\mathfrak {gl}(n) $-modules of index 2, arXiv:1606.03394v1.
\bibitem[FO]{FO} V. Futorny, S. Ovsienko, Fibers of characters in Gelfand-Tsetlin categories,  {Trans. Amer. Math. Soc.} {366} (2014), 4173--4208.
\bibitem[FOS]{FOS}  V.Futorny,  S.Ovsienko, M.Saorin,
Torsion theories induced from commutative subalgebras, Journal of
Pure and Applied Algebra (2011), 2937-2948.
 \bibitem[GLZ]{GLZ} X. Guo, G. Liu and K. Zhao, Irreducible Harish-Chandra modules over extended
Witt algebras, Ark. Mat., 52 (2014), 99-112.
 \bibitem[GZ]{GZ} X.Guo and K.Zhao, Irreducible weight modules over Witt algebras, Proc. Amer. Math. Soc.,
139(2011), 2367-2373.
\bibitem[L1]{L1} T. A. Larsson, Multi dimensional Virasoro algebra, Phys.
Lett.,
B 231, 94-96(1989).
\bibitem[L2]{L2} T. A. Larsson, Central and
non-central extensions of multi-graded Lie algebras, J. Phys., A 25,
1177-1184(1992).
\bibitem[L3]{L3} T. A. Larsson,  Conformal fields:
A class of representations of Vect (N), Int. J. Mod. Phys., A 7,
6493-6508(1992).
\bibitem[L4]{L4} T. A. Larsson, Lowest energy
representations of non-centrally extended diffeomorphism algebras,
Commun. Math. Phys., 201, 461-470(1999).
\bibitem[L5]{L5} T. A.
Larsson, Extended diffeomorphism algebras and trajectories in jet
space, Commun. Math. Phys., 214, 469-491(2000).
\bibitem[LZ]{LZ} G. Liu, K. Zhao, New irreducible weight  modules over Witt algebras with
infinite dimensional weight spaces,  Bulletin of the Lond. Math. Soc., 47(2015), 789-795.
\bibitem[M]{M} O. Mathieu; Classification of irreducible weight modules, Ann. Inst.
Fourier (Grenoble) 50 (2000), no. 2, 537-592.
\bibitem[MZ1]{MZ1}  V. Marzuchuk, K. Zhao, Characterization of simple highest weight
modules, Can. Math. Bull. 56(3), 606-614 (2013)
\bibitem[MZ2]{MZ2} V. Marzuchuk and K. Zhao, Supports of weight modules over Witt
algebras, Proc. Roy. Soc. Edinburgh Sect., A 141(2011), no. 1,
155-170.
\bibitem[N]{N} J. Nilsson, Simple $\sl_{n+1}$-module structures on $U(h)$,  J. Algebra, {\bf 424}  (2015),  294-329.
\bibitem[Ov]{Ov} S. Ovsienko, Finiteness
statements for Gelfand-Zetlin modules, {Third International Algebraic Conference in the Ukraine (Ukrainian)},  Natsional. Akad. Nauk Ukrainy, Inst. Mat., Kiev, (2002), 323--338.
\bibitem[Ru]{Ru} A.~N.~Rudakov,
{Irreducible representations of infinite-dimensional Lie algebras
of Cartan type}, Izv. Akad. Nauk SSSR Ser. Mat. {\bf 38} (1974),
835-866.
\bibitem[Sh]{Sh} G. Shen, Graded modules of graded Lie algebras of
Cartan type. I. Mixed products of modules,  Sci. Sinica Ser., A {\bf
29}  (1986),  no. 6, 570-581.
\bibitem[TZ]{TZ}H. Tan, K. Zhao,  $\mathcal{W}_n^+$ and
$\mathcal{W}_n$-module structures on $U(h)$, J. Algebra, {\bf 424}  (2015),  357-375.
\bibitem[Z]{Z} K. Zhao, Weight modules over generalized Witt algebras with
1-dimensional weight spaces, Forum Math., Vol.16, No.5, 725-748.

\end{thebibliography}
\end{document}